\documentclass[a4paper]{siamltex}

\usepackage{amsmath,amssymb,amsfonts}
\usepackage{comment,url}
\usepackage{pgfplots}
\usepackage{pgfplotstable}
\usepackage{booktabs}
\usepackage{tabularx}

\newtheorem{remark}{Remark}

\usepackage{comment,url}
\usepackage{pgfplots}
\usepackage{pgfplotstable}
\usepackage{booktabs}
\usepackage{tabularx}

\newcommand{\norm}[1]{\lVert #1 \rVert}
\newcommand{\Z}{\mathbb Z}
\newcommand{\w}{{_{_\mathcal W}}}
\newcommand{\ww}{{_{_{\mathcal W_1}}}}

\newcommand{\f}{{_{_\mathcal F}}}
\newcommand{\aqt}{{_{_\mathcal{CQT}}}}
\newcommand{\qt}{{_{_\mathcal{QT}}}}
\newcommand{\wt}{\widetilde}
\newcommand{\trunc}{\hbox{trunc}}
\pgfplotstableset{
  every head row/.style={before row=\toprule,after row=\midrule},
  clear infinite
}

\begin{document}
\title{Semi-Infinite Quasi-Toeplitz Matrices with Applications to QBD Stochastic Processes\thanks{Work supported by GNCS of INdAM.}}
\author{Dario A. Bini
\thanks{Dipartimento di Matematica, Universit\`a di Pisa, ({\tt dario.bini@unipi.it})}
\and 
Stefano Massei
\thanks{Scuola Normale Superiore, Pisa ({\tt stefano.massei@sns.it})}
\and
Beatrice Meini
\thanks{Dipartimento di Matematica, Universit\`a di Pisa, ({\tt beatrice.meini@unipi.it})}}

\maketitle

\begin{abstract}
Denote by $\mathcal{W}_1$ the set of complex valued functions of the form
$a(z)=\sum_{i=-\infty}^{+\infty}a_iz^i$ which are continuous on the unit circle, and such that $\sum_{i=-\infty}^{+\infty}|ia_i|<\infty$. We call CQT matrix a quasi-Toeplitz matrix $A$, associated with a continuous symbol $a(z)\in\mathcal W_1$,  of the form $A=T(a)+E$, where
  $T(a)=(t_{i,j})_{i,j\in\Z^+}$ is the semi-infinite Toeplitz matrix
  such that $t_{i,j}=a_{j-i}$, for $i,j\in\mathbb Z^+$, and
  $E=(e_{i,j})_{i,j\in\Z^+}$ is a semi-infinite matrix such that
  $\sum_{i,j=1}^{+\infty}|e_{i,j}|$ is finite.
    We prove that the class of CQT matrices is a Banach algebra with a suitable sub-multiplicative matrix norm
  $\|\cdot\|$. We introduce a finite representation of CQT matrices together with
  algorithms which implement elementary matrix operations.  An
  application to solving quadratic matrix equations of the kind
  $AX^2+BX+C=0$, encountered in the solution of Quasi-Birth and Death
  (QBD) stochastic processes with a denumerable set of phases, is
  presented where $A,B,C$ are CQT matrices.
\end{abstract}

\maketitle

\section{Introduction}
Toeplitz matrices, i.e., matrices of the kind $T=(t_{i,j})$ such that
$t_{i,j}=a_{j-i}$ for some sequence $\{a_k\}_{k\in\Z}$, are
encountered in many applications.
In certain stochastic processes, like in the analysis of random walks
in the quarter plane \cite{fayolle}, or in the analysis of the tandem Jackson queue models \cite{jackson1957networks}, 
\cite{Sakuma-Miyazawa}, one 
typically encounters semi-infinite Toeplitz matrices, where the
indices of the entries range in the set $\mathbb Z^+$ of positive
integers \cite{Sakuma-Miyazawa}, \cite{miyazawa2011light}, \cite{kobayashi2012revisit}, \cite{Takahashi}. In fact,
these applications are modeled by a  block tridiagonal generator $Q$  of the form
\[
Q=\begin{bmatrix}\widehat A_0&\widehat A_1\\
A_{-1}&A_0&A_1\\
      &A_{-1}&A_0&A_1\\
      &&\ddots&\ddots&\ddots\end{bmatrix}
\]
where the blocks $A_{-1}$, $A_0$, $A_1$, $\widehat A_0$, $\widehat
A_1$ are  semi-infinite tridiagonal and quasi Toeplitz matrices. More
specifically, they can be written as the sum of a tridiagonal Toeplitz matrix and a correction, that is,
$
A_i=\hbox{trid}(\mu_i,\sigma_i,\nu_i)+F_i
$, $i=-1,0,1$,
$
\widehat A_i=\hbox{trid}(\widehat \mu_i,\widehat \sigma_i,\widehat \nu_i)+\widehat F_i
$, $i=0,1$. Here,
 $\hbox{trid}(\mu,\sigma,\nu)$ denotes a semi-infinite tridiagonal
Toeplitz matrix with sub-diagonal, diagonal and super-diagonal entries
$\mu,\sigma,\nu$, respectively, and $F_i$, $\widehat F_i$ denote
matrices which are  possibly nonzero only in the entries of indices
$(1,1)$ and $(1,2)$.
Generators $Q$ in block tridiagonal form characterize the very wide class of Quasi-Birth-and-Death (QBD) processes \cite{lr:book}.
Observe that the Toeplitz part $\hbox{trid}(\mu,\sigma,\nu)$ is uniquely
determined by the Laurent polynomial 
$a(z)=z^{-1}\mu+ \sigma+z\nu$
which is called the {\em symbol} associated with the Toeplitz matrix.
 
An important problem in the analysis of a QBD process is to compute the minimal non-negative solutions $G$ and $R$ of the associated matrix equations  $A_{-1}+A_0X+A_1X^2=0$ and $X^2A_{-1}+XA_0+A_1=0$, respectively, see for instance \cite{lr:book}, \cite{bini2005numerical}.
If the matrix size is finite, the algorithms of Cyclic Reduction and of Logarithmic Reduction can be effectively used to solve these equations. Other techniques based on fixed point iterations can be used as well.
Theoretically, these algorithms can also be used  in the case where matrices are
semi-infinite.  However, in this case, difficult computational issues are encountered  because performing arithmetic operations between Toeplitz
matrices, generally causes the loss of
sparsity and of the Toeplitz structure. This creates the nontrivial problem of storing infinite matrices, with apparently no structure, by means of a finite number of parameters. Another interesting issue is to figure out if the solutions $R$ and $G$ of the associated matrix equations share, in some form, part of the Toeplitz structure.

Let $\mathcal W$ be the Wiener class formed by the complex valued functions $a(z)=\sum_{i\in\mathbb Z}a_iz^i$, defined on the unit circle, such that $\|a\|\w:=\sum_{i\in\Z}|a_i|$ is finite. Moreover, define $\mathcal W_1\subset\mathcal W$ the subclass of functions $a(z)$ which are continuous and such that $a'(z)=\sum_{i\in\mathbb Z}ia_iz^{i-1}\in\mathcal W$.

In this work we introduce the class $\mathcal{QT}$ of semi-infinite Quasi-Toeplitz (QT) matrices, that is, matrices 
of the form $A=T(a)+E$ where $T(a)$ is the Toeplitz matrix
associated with the symbol $a(z)\in\mathcal W$, and the correction $E=(e_{i,j})$ is such that $\|E\|\f:=\sum_{i,j\in\Z^+}|e_{i,j}|$ is finite as well. This class provides a generalization of the structure encountered in QBD problems where the symbol $a(z)$ is a Laurent polynomial and the correction matrix has only a finite number of nonzero entries.

We prove that $\mathcal{QT}$ with the norm
$\|A\|\qt=\|a\|\w+\|E\|\f$ 
is a Banach space but is not closed
under matrix product. Instead, the subclass $\mathcal{CQT}$ formed by matrices in
$\mathcal {QT}$ associated with a continuous symbol $a(z)\in\mathcal W_1$, forms a Banach algebra with the norm
$\|\cdot\|\aqt$  such that $\|T(a)+E\|\aqt:=\|a\|\w+\|a'\|\w+\|E\|\f$ and
$\|AB\|\aqt\le\|A\|\aqt\|B\|\aqt$
for any $A,B\in\mathcal{CQT}$.

 A nice property of the class $\mathcal{QT}$ is that for any $A\in\mathcal{QT}$ and for any $\epsilon>0$ there
 exists $B\in\mathcal{CQT}$, determined by a {\em finite number of parameters}, such that $\|A-B\|\qt\le\epsilon$. This fact allows us to represent any matrix in $\mathcal{QT}$ with a finite number of parameters up to an arbitrarily small error in the QT-norm. We also
 introduce algorithms that execute the arithmetic operations between
 CQT matrices, and provide their Matlab implementation.
 This way, we may extend standard algorithms, valid for finite matrices, to the case of CQT matrices. In particular, we show how the algorithm of Cyclic Reduction \cite{bini2009cyclic}  can be adapted to solve
quadratic matrix equations
 of the kind $AX^2+BX+C=0$, where $A,B,C\in\mathcal{CQT}$, which are encountered in QBD processes modeling random walks in
 the quarter plane \cite{fayolle}, and  the Jackson Tandem Queue
 \cite{jackson1957networks}, \cite{taylor}. 
Some numerical experiments performed with a set of problems presented in 
\cite{taylor} show the effectiveness of our approach.

The decomposition of a matrix as the sum of a Toeplitz part plus a correction has been used in the literature in different contexts. For instance, in \cite[Example 2.28]{bottcher2005spectral} 
matrices of the form $T(a)+E$ are considered where $E$ is a compact operator with finite $\ell^2$ operator norm. It is worth pointing out that the boundedness of $\|E\|_2$ does not imply that $\|E\|\f<\infty$ which is required for our computational goals. In the framework of Toeplitz preconditioning and in the analysis of asymptotic spectral properties of finite Toeplitz sequences the decomposition of a Toeplitz matrix in the form $T(a)+E+R$ is considered where $T(a)$ is banded, $E$ and $R$ are corrections of small norm and small rank, respectively; among the many papers on this subject we cite  \cite{dibenede}, \cite{serra}, \cite{tyrty} with the many related references.

The analysis and the tools presented in this paper can be used for the
effective numerical computation of matrix functions expressed by means
of a Taylor expansion, like the exponential function, or expressed by
means of an integral representation. These applications are shown in detail
in \cite{bmm1}, \cite{bm:exp}. In particular, in \cite{bmm1} it is shown how this machinery can be extended to the case where matrices are finitely large. The problem of computing the exponential of finite Toeplitz matrices has been recently investigated in \cite{kressner} relying on the concept of displacement rank.

It is worth pointing out that the definition of matrix function of a CQT matrix $A$, as well as the algorithms implementing the CQT matrix arithmetic, are somehow related to the decay properties of the coefficients of the matrices $T(f(a))$ and $E_{f(a)}$ such that $f(A)=T(f(a))+E_{f(a)}$, and also to the numerical rank of the product of two Hankel matrices associated with analytic functions. The analysis of decay properties of matrix functions and of the singular values of some structured matrices, having a displacement rank structure, have recently received much interest and have been investigated in \cite{becker}, \cite{BeBo14}, \cite{BeSi15}, \cite{posi16}.

The paper is organized as follows. In Section \ref{sec:prel} we recall some preliminary properties which are needed in our analysis. In Section \ref{sec:main} we prove that $\mathcal{CQT}$ is a Banach algebra. In Section \ref{sec:arith} we describe the way in which matrix operations can be defined and implemented in $\mathcal{CQT}$ and report a few notes on our Matlab implementation of CQT-arithmetic. In Section \ref{sec:qbd} 
we present an application to solving a matrix equation encountered in QBD stochastic processes together with the results of some numerical experiments which confirm the effectiveness of the class $\mathcal{CQT}$. Section \ref{sec:conc} draws the conclusions.

\section{Notation and preliminaries} \label{sec:prel}

Denote by $\mathbb T=\{z\in\mathbb C:~ |z|=1\}$ the unit circle in the
complex plane, and by $\mathcal W$ the Wiener class formed by the
functions $a(z)=\sum_{i=-\infty}^{+\infty}a_iz^i:\mathbb T\to\mathbb
C$ such that $\sum_{i=-\infty}^{+\infty}|a_i|<+\infty$.  Recall that
$\mathcal W$ is a Banach algebra, that is, a vector space closed under
multiplication, endowed with the norm $\|a\|\w:=\sum_{i\in\mathbb Z}|a_i|$ which makes the space complete and
such that $\|ab\|\w\le\|a\|\w\|b\|\w$  for any $ a(z),b(z)\in\mathcal W$.
We refer the reader to the first chapter of the book \cite{bottcher2005spectral} for more details.

In the following, we denote by $a^+(z)$ and by $a^-(z)$ the power series
defined by the coefficients of $a(z)$ with positive and with negative powers, respectively, that is, 
$a^+(z)=\sum_{i\in\mathbb Z^+}a_iz^i$ and 
$a^-(z)=\sum_{i\in\mathbb Z^+}a_{-i}z^{i}$, so that $a(z)=a_0+a^+(z)+a^-(z^{-1})$. 
We associate with the Laurent series $a(z)$, and with the power series $b(z)=\sum_{i=0}^\infty b_iz^i$  the 
following semi-infinite matrices
\[
\begin{array}{lll}
T(a)=(t_{i,j})_{i,j}, & t_{i,j}=a_{j-i},\\[1ex]
H(b)=(h_{i,j})_{i,j},& h_{i,j}=b_{i+j-1},&\quad i,j\in\Z^+,
\end{array}
\]
respectively. Observe that $T(a)$ is a Toeplitz matrix while $H(b)$ is  Hankel.

Finally, denote by $\mathcal F$
the class of semi-infinite matrices $F=(f_{i,j})_{i,j\in\mathbb Z^+}$
such that $\|F\|\f:=\sum_{i,j\in\mathbb Z^+}|f_{i,j}|$ is finite.
The norm that we use in this case is just the 1-norm if we look at the matrix $F$ as an infinite vector.

Observe that $\mathcal F$ is a vector space, closed under rows-by-columns multiplication,  and $\|F\|\f$ is a norm over $\mathcal F$ which is endowed of the sub-multiplicative property. In the following, we write $(\mathcal F,\|\cdot\|\f)$ to denote the linear space $\mathcal F$ endowed with the norm $\|\cdot\|\f$. We have the following:

\begin{lemma}\label{lem:1}
$(\mathcal F,\|\cdot\|\f)$ equipped with matrix sum and multiplication is a Banach algebra over $\mathbb C$.
\end{lemma}
\begin{proof}
We need to show that given $E,F\in\mathcal F$ and $\alpha\in\mathbb C$ it holds
\begin{itemize}
\item[(i)] $\alpha E\in\mathcal F$,
\item[(ii)] $E+F\in\mathcal F$,
\item[(iii)] $EF\in\mathcal F$ and  $\norm{EF}\f\leq \norm{E}\f\norm{F}\f$,
\item[(iv)] $(\mathcal F,\|\cdot\|\f)$ is a complete metric space.
\end{itemize}
Clearly, $\sum_{i,j\in\mathbb Z^+}|\alpha e_{i,j}|=|\alpha|\sum_{i,j\in\mathbb Z^+}|e_{i,j}|<+\infty$ 
which proves (i). By the triangular inequality one obtains that  $\sum_{i,j\in\mathbb Z^+}|e_{i,j}+f_{i,j}|\le\sum_{i,j\in\mathbb Z^+}|e_{i,j}|+\sum_{i,j\in\mathbb Z^+} |f_{i,j}|<+\infty$ which implies (ii). If $H=EF=(h_{i,j})$ then  $h_{i,j}=\sum_{r\in\mathbb Z^+}e_{i,r}f_{r,j}$ so that, defining $\alpha_r=\sum_{i\in\mathbb Z^+}|e_{i,r}|$, and $\beta_r=\sum_{j\in\mathbb Z^+}|f_{r,j}|$, for the quantity $\|EF\|\f=\sum_{i,j\in\mathbb Z^+} |h_{i,j}|$ we have
\[
\|EF\|\f\le\sum_{i,j,r\in\mathbb Z^+}
|e_{i,r}|\cdot|f_{r,j}|=
\sum_{r\in\mathbb Z^+}\alpha_r\beta_r\le
\left(\sum_{r\in\mathbb Z^+}\alpha_r\right)\left(\sum_{r\in\mathbb Z^+}\beta_r\right)=\|E\|\f\cdot\|F\|\f,
\]
which shows (iii). Finally, we observe that any matrix $E\in\mathcal F$ can be viewed as a vector $v=(v_k)_{k\in\mathbb Z^+}$ obtained by suitably ordering the entries $e_{i,j}$. 
Moreover, the norm $\|\cdot\|\f$ corresponds to the $\ell^1$ norm in the space of infinite vectors  having finite sum of their moduli. This way, the space $\mathcal F$ actually coincides with $\ell^1$, which is a Banach space.
Thus, we get (iv).
\end{proof}

Observe that the condition $\|F\|\f<+\infty$ implies that for any $\epsilon>0$ there exists an integer $k>0$ such that $\sum_{i,j\ge k}|f_{i,j}|<\epsilon$, that is, the entries of the matrix $F$ decay to zero as $i, j\to\infty$ so that $F$ can be approximated with an arbitrarily small error by a matrix with finite support. This property is of fundamental importance in order to represent the matrix $F$ with a finite number of parameters up to an error which is  negligible with respect to the round-off error.

Any semi-infinite matrix $S=(s_{i,j})_{i,j\in\mathbb Z^+}$ can be viewed as a linear operator, acting on semi-infinite vectors $v=(v_i)_{i\in\mathbb Z^+}$, which maps the vector $v$ onto the vector $u$ such that
$u_i=\sum_{j\in\mathbb Z^+}s_{i,j}v_j$, provided that the results of the summations are finite. 

Indeed, the matrices $F\in\mathcal F$ define linear operators on the space $\ell^1$ of semi-infinite  vectors $v=(v_i)$ such that $\|v\|_1=\sum_{i\in\mathbb Z^+}|v_i|$ is finite, since
\[
\sum_{i\in\mathbb Z^+}|\sum_{j\in\mathbb Z^+}f_{i,j}v_j|\le
\sum_{i,j\in\mathbb Z^+}|f_{i,j}v_j|\le \sum_{i,j\in\mathbb Z^+}|f_{i,j}|\cdot\sup_k |v_k|
\]
which is finite as the product of two finite terms.

For any integer $p\ge 1$, we may wonder if also the matrices $T(a)$, $H(a^+)$ and $H(a^-)$
define linear operators acting on the Banach space $\ell^p$ formed by vectors $v$ such that
the $\ell^p$ norm $\| v\|_p=(\sum_{i\in\mathbb Z^+}|v_i|^p)^{1/p}$ is finite. In this case we may evaluate the $p$-norm of the operator $S$ (operator norm) as $\|S\|_p:=\sup_{\|v\|_p=1}\|Sv\|_p$. The answer to this question is given by the following result of \cite{bottcher2005spectral}
which relates the matrix $T(a)T(b)$ with $T(ab)$, $H(a^-)$ and $H(a^+)$.

\begin{theorem}\label{th:1}
For $a(z),b(z)\in\mathcal W$ let $c(z)=a(z)b(z)$. Then we have
\[
T(a)T(b)=T(c)-H(a^-)H(b^+).
\]
Moreover, for any $a(z)\in\mathcal W$  and for any $p\ge 1$, including $p=\infty$, we have 
\[
\|T(a)\|_p\le \|a\|\w,\quad \|H(a^-)\|_p\le\|a^-\|\w,\quad \|H(a^+)\|_p\le\|a^+\|\w.
\] 
\end{theorem}

A direct consequence of the above result is that the product of two Toeplitz matrices can be written as a Toeplitz matrix plus a correction whose $\ell^p$-norm is bounded by $\|a\|\w\|b\|\w$.

A similar property holds for matrix inversion in the case where the function $a(z)$ is nonzero for $|z|=1$ and its winding number is zero.
In fact, in this case we may apply another classical result (we refer to the book \cite{boett-grudsky-1} for more details) which relates the invertibility of the operator $T(a)$ to the winding number $\kappa$ of $a(z)$, that is, the (integer) number of times that the complex number $a(\cos\theta+
\mathbf{i} \sin \theta)$, where $\mathbf{i}^2=-1$,  winds around the origin as $\theta$ moves from $0$ to $2\pi$.

\begin{theorem} [Gohberg 1952]\label{th:goh} Let $a(z)$ be  a continuous function from $\mathbb T$ in $\mathbb C$. Then the linear operator $T(a)$ is invertible if and only if the winding 
number of $a(z)$ is zero and $a(z)$ does not vanish on $\mathbb T$.
\end{theorem}

Thus, under the assumptions of the above theorem, it follows that
$T(a)$ is invertible and  we have 
$
T(a)^{-1}=T(a^{-1})+E,
$
where $\|E\|_p$ is bounded from above by a constant \cite[Proposition 1.18]{bottcher2005spectral}.

In the analysis that we are going to perform in the next section, the
above properties concerning the $\ell^p$ norms are very useful, but
are not enough to arrive at an algorithmic implementation concerning
Toeplitz and quasi-Toeplitz matrices. In fact, our request is to write
the product and the inverse of Toeplitz matrices as a Toeplitz matrix
plus a correction whose entries have a decay along the diagonals. In
fact, in this case, the correction can be approximated with any
precision by using a finite number of parameters.
Observe that, for the matrix product, this property is satisfied if $E=H(a^-)H(b^+)\in\mathcal F$ in view of Theorem \ref{th:1}.

Finally, we recall a result concerning the Wiener-Hopf factorization of $a(z)$ which will be useful next.

\begin{theorem}\label{th:whf}
Let $a(z)\in\mathcal W$ be a function which does not vanish for $z\in\mathbb T$ and such that its winding number is $\kappa$. Then $a(z)$ admits the Wiener-Hopf factorization
\[
a(z)=u(z)z^\kappa\ell(z),
\]
where $u(z)=\sum_{i=0}^\infty u_iz^i$, $\ell(z)=\sum_{i=0}^\infty \ell_iz^{-i}$ are in $\mathcal W$ and $u(z)$, $\ell(z^{-1})$ do not vanish in the closed unit disk.
If $\kappa=0$ the factorization is said canonical.  
\end{theorem}

\section{Quasi-Toeplitz matrices}\label{sec:main}
In this section we introduce the classes  of quasi-Toeplitz  matrices and  analyze their properties.

\begin{definition}\label{def:qt}
We say that the semi-infinite matrix $A$  is a {\em quasi-Toeplitz matrix (QT-matrix)} if it can be written in the form
\[
A=T(a)+E,
\]
where $a(z)=\sum_{i=-\infty}^{+\infty} a_iz^i$ is in the Wiener class,
and $E=(e_{i,j})\in\mathcal F$.  We refer to $T(a)$ as the Toeplitz
part of $A$, and to $E$ as the correction. We denote by $\mathcal{QT}$ the class of QT-matrices.
Moreover we define the following norm on $\mathcal{QT}$
\[
\norm{T(a)+E}\qt:=\norm{a}\w+\norm{E}\f.
\] 
\end{definition}

Observe that given $A\in\mathcal{QT}$ there is a unique way to decompose it in the sense of Definition~\ref{def:qt}. In fact, suppose by contradiction that there exist $a_1(z),a_2(z)\in\mathcal W$ and $E_1,E_2\in\mathcal F$ with  $a_1\neq a_2$ and $E_1\neq E_2$  such that
\[
A=T(a_1)+E_1=T(a_2)+E_2.
\]
Then we should have $E_1-E_2=T(a_2)-T(a_1)=T(a_2-a_1)$, hence
$\| E_1-E_2\|_{\mathcal{F}}=\| T(a_2-a_1)\|_{\mathcal{F}}$. On the other hand, since $T(a_2-a_1)\neq 0$ we have $\| T(a_2-a_1)\|_{\mathcal{F}}=\infty$, which contradicts the fact that $E_1-E_2\in\mathcal F$.

\begin{lemma}
The set $\mathcal{QT}$ endowed with the norm
$
\|\cdot\|\qt
$
is a Banach space.
\end{lemma}
\begin{proof}
The set of quasi Toeplitz matrices is clearly isomorphic to the direct sum 
$
\mathcal{QT}\simeq \mathcal W \oplus \mathcal F.
$
Since both $\mathcal W$ and $\mathcal F$ are Banach spaces the composition of the $1$-norm of $\mathbb R^2$ with the vector valued function $T(a)+E\to (\norm{a}\w,\norm{E}\f)$ makes $\mathcal W \oplus \mathcal F$ a complete metric space.
\end{proof}

The class $\mathcal{QT}$ clearly includes all the matrices encountered in QBD processes,  formed by
a banded Toeplitz part, 
and  by a correction $E$ such that $e_{i,j}=0$
for $i,j>k$ for some integer $k$.

The goal of this section is to prove that the subclass of QT-matrices associated with continuous symbols $a(z)$ such that $a'(z)\in\mathcal W$  form a normed matrix algebra, i.e., a vector space
closed under matrix multiplication.  To this end, it is useful to introduce the following sub-algebra of $\mathcal W$.

\begin{definition}
We denote $\mathcal W_1=\{a(z)\in\mathcal W:\quad a(z)\hbox{ continuous, and }a'(z)\in\mathcal W\}$, and define the norm
\[
\|a\|\ww=\|a\|\w+\|a'\|\w.
\]
\end{definition}

We recall that   $\mathcal W_1$ is a Banach algebra with the norm $\|a\|\ww$, see \cite{bottcher2012introduction}.

\begin{definition}\label{def:aqt}
We call CQT-matrix, any matrix $T(a)+E\in\mathcal{QT}$  such that the  symbol $a(z)\in\mathcal W_1$.  We denote by $\mathcal{CQT}$ the subset of $\mathcal{QT}$ formed by CQT-matrices. Moreover, we define the following norm in $\mathcal{CQT}$:
\[
\norm{T(a)+E}\aqt:=\norm{a}\ww+\norm{E}\f.
\]
\end{definition}

Next, we provide a few results which are useful to prove that $\mathcal{CQT}$ is a Banach algebra.
The following lemma shows that the product of two semi-infinite Toeplitz matrices associated with symbols in $\mathcal W_1$ belongs to $\mathcal{CQT}$.
 
\begin{lemma}\label{lem:2}
  Let $a(z),b(z)\in\mathcal W_1$  and set $c(z)=a(z)b(z)$. Then 
  $T(a)T(b)=T(c)+E_c$ where $E_c\in\mathcal F$, moreover,
  \[
  \|E_c\|\f\le \|H(a^-)\|\f\cdot\|H(b^+)\|\f=\sum_{i\in\mathbb Z^+}i|a_{-i}| \sum_{i\in\mathbb Z^+}i|b_{i}|.
  \]
\end{lemma}
\begin{proof} From Theorem \ref{th:1} we deduce that $T(a)T(b)=T(c)+E_c$ where we set $E_c=-H(a^-)H(b^+)$.
Let us prove that $H(a^-),H(b^+)\in\mathcal F$. 
We have $\|H(b^+)\|\f=\sum_{i,j\in\mathbb Z^+}|b_{i+j-1}|$. Setting $k=i+j-1$ we may write 
$\|H(b^+)\|\f=\sum_{k\in\mathbb Z^+}k|b_{k}|$ which is finite since $b(z)\in\mathcal W_1$.
The same argument applies to $H(a^-)$. In view of Lemma \ref{lem:1}, $\mathcal F$ is a normed matrix algebra therefore $\|E_c\|\f\le\|H(a^-)\|\f\cdot\|H(b^+)\|\f<+\infty$.
\end{proof}

\begin{remark}\label{rem:1}\rm
Observe that the quantities $\sum_{i\in\mathbb Z^+}i|a_{-i}|$ and  $\sum_{i\in\mathbb Z^+}i|b_{i}|$ coincide with the $\mathcal W$-norms of the first derivatives of the functions $a^-(z)$ and $b^+(z)$, respectively. This way we may rewrite the bound given in Lemma \ref{lem:2} as
\begin{equation}\label{eq:boundEc}
\|E_c\|\f\le\|(a^-)'\|\w\|(b^+)'\|\w\le\|a'\|\w\|b'\|\w.
\end{equation}
 \end{remark}
 
The condition $a(z),b(z)\in\mathcal W_1$ is needed to
prove Lemma \ref{lem:2}, as it is demonstrated by the following
example. Consider the case where
$a(z)=\sum_{i=0}^{+\infty}a_{-i}z^{-i}$,
$b(z)=\sum_{i=0}^{+\infty}b_iz^i$, $a_{-i}=b_i=i^{-3/2}$.  Clearly
$a(z),b(z)\in\mathcal W$ but $a(z)'$ and $b(z)'$ are not in $\mathcal W$ since $\sum_{i\in\mathbb Z^+}ia_{-i}$ and $\sum_{i\in\mathbb Z^+}ib_{i}$ are
not convergent. Moreover,
\[
\|H(a^-)H(b^+)\|\f=\sum_{i,j\in\mathbb
  Z^+}\sum_{r=0}^{+\infty}\frac
1{(i+r)^{3/2}}\frac1{(r+j)^{3/2}}=\sum_{r=0}^{+\infty}\left(\sum_{k=r+1}^{+\infty}\frac
  1{k^{3/2}}\right)^2.
\]
 This is the sum of the squares of the remainders of
  the series $\sum_{i=1}^{+\infty}\frac 1{i^{3/2}}$. This sum diverges since these remainders behave like
  $\int_r^{+\infty}\frac 1{x^{3/2}}dx=\frac{2}{\sqrt r}$.

Now we can prove the main result of this section which states that $\mathcal{CQT}$ is closed under multiplication.
 
\begin{theorem}\label{thm:aqt-mult}
  Let $A,B\in \mathcal{CQT}$, where $A=T(a)+E_a$, $B=T(b)+E_b$. Then we have $C=AB=T(c)+E_c\in\mathcal{CQT}$ with
  $c(z)=a(z)b(z)$.  Moreover,
\[
\|E_c\|\f\le \|H(a^-)\|\f\cdot\|H(b^+)\|\f+\|a\|\w \|E_b\|\f+\|b\|\w\|E_a\|\f+\|E_a\|\f\cdot\|E_b\|\f.
\]
\end{theorem}
\begin{proof}
 We have
$
C=AB=(T(a)+E_a)(T(b)+E_b)  
$.
Applying Theorem \ref{th:1} yields
\[
C=T(c)-H(a^-)H(b^+)+T(a)E_b+E_aT(b)+E_aE_b=:T(c)+E_c,
\]
where
\begin{equation}\label{eq:Ec}
E_c=-H(a^-)H(b^+)+T(a)E_b+E_aT(b)+E_aE_b.
\end{equation}
Therefore, it is sufficient to prove that $\|E_c\|\f$ is finite.
From Lemmas~\ref{lem:2} and \ref{lem:1} it follows that both $\|H(a^-)H(b^+)\|\f$ and $\|E_aE_b\|\f$ are finite.  It remains to show that  $\|E_aT(b)\|\f$ and $\|T(a)E_b\|\f$ are finite. We prove this property only for $\|T(a)E_b\|\f$ since the boundedness of the other matrix norm follows by transposition. In fact, for any $F\in\mathcal F$ one has $\|F\|\f=\|F^T\|\f$ and $T(a)^T=T(\widehat a)$ where $\widehat a(z)=a(z^{-1})$ and $\|a\|\w=\|\widehat a\|\w$. Denote $H=T(a)E_b=(h_{i,j})$ and $E_b=(e_{i,j})$. 
We have $h_{i,j}=\sum_{r=1}^{+\infty}a_{r-i}e_{r,j}$ 
 so that
\[
\|H\|\f=\sum_{i,j\in\mathbb Z^+}|h_{i,j}|\le\sum_{i,j\in\mathbb Z^+}
\sum_{r=1}^{+\infty}|a_{r-i}e_{r,j}|.
\]
Substituting $k=r-i$ yields
\[
\|H\|\f\le\sum_{k\in\mathbb Z}|a_k|\sum_{j=1}^{+\infty}\sum_{i=-k+1}^{+\infty}|e_{k+i,j}|.
\]
Since  $\sum_{j=1}^{+\infty}\sum_{i=-k+1}^{+\infty}|e_{k+i,j}|=
\sum_{j=1}^{+\infty}\sum_{i=1}^{+\infty}|e_{i,j}|=\|E_b\|\f$ for any $k$, we have
\[
\|H\|\f\le\sum_{k\in\mathbb Z}|a_k|\|E_b\|\f=\|a\|\w\|E_b\|\f<+\infty.
\]
Thus, taking norms in \eqref{eq:Ec} yields
\[
\| E_c\|\f\le \|H(a^-)\|\f\cdot\|H(b^+)\|\f+\|a\|\w\|E_b\|\f+\|E_a\|\f\cdot\|b\|\w+\|E_a\|\f\cdot\|E_b\|\f
\]
which completes the proof.
\end{proof}

Observe that in view of Remark \ref{rem:1} we may write
\begin{equation}\label{eq:in1}
\| E_c\|\f\le \|a'\|\w\|b'\|\w+\|a\|\w\|E_b\|\f+\|E_a\|\f\cdot\|b\|\w+\|E_a\|\f\cdot\|E_b\|\f.
\end{equation}

Now, our next goal is to prove that the class $\mathcal{CQT}$ is a Banach  algebra.

\begin{theorem}\label{thm:algebra}
The class $\mathcal{CQT}$ equipped with the norm $\norm{\cdot}\aqt$  is a Banach algebra over $\mathbb C$. Moreover $\|AB\|\aqt\le\|A\|\aqt\|B\|\aqt$ for any matrices $A,B\in\mathcal{CQT}$.
\end{theorem}
\begin{proof}
Theorem~\ref{thm:aqt-mult} ensures the closure of $\mathcal{CQT}$ under matrix multiplication. To prove the sub-multiplicative property of the norm, i.e., \[\|AB\|\aqt\le\|A\|\aqt\cdot\|B\|\aqt\] for any $A,B\in\mathcal{CQT}$, $A=T(a)+E_a$, $B=T(b)+E_b$,  observe that 
\begin{equation}\label{eq:in2}
\begin{split}
\| ab\|\ww=&\|ab\|\w+\|(ab)'\|\w=\|ab\|\w+\|a'b+ab'\|\w\\
\le&\|a\|\w\|b\|\w+\|a'\|\w\|b\|\w+\|a\|\w\|b'\|\w.
\end{split}
\end{equation}
Since
$\|AB\|\aqt=\|ab\|\ww+\|E_c\|\f$, for $c(z)=a(z)b(z)$, and where $E_c$ is defined as in Theorem~\ref{thm:aqt-mult}, by applying \eqref{eq:in1} and
\eqref{eq:in2}   we obtain
\[\begin{split}
\|AB\|\aqt&\le \|ab\|\ww +\|a'\|\w\|b'\|\w+\|a\|\w\|E_b\|\f
 +\|b\|\w\|E_a\|\f+\|E_a\|\f\|E_b\|\f\\
&\le \|a\|\w\|b\|\w+\|a'\|\w\|b\|\w+\|a\|\w\|b'\|\w \\ 
&\quad +\|a'\|\w\|b'\|\w
+\|a\|\w\|E_b\|\f+\|b\|\w\|E_a\|\f+\|E_a\|\f\|E_b\|\f\\
&=(\|a\|\w+\|a'\|\w)(\|b\|\w+\|b'\|\w)\\
&\quad +\|a\|\w\|E_b\|\f+\|b\|\w\|E_a\|\f
 +\|E_a\|\f\|E_b\|\f\\
&\le (\|a\|\ww+ \|E_a\|\f)(\|b\|\ww+ 
\|E_b\|\f)\\&=\|A\|\aqt\|B\|\aqt.
\end{split}
\]
Concerning the completeness, observe that the set of CQT matrices is isomorphic to the direct sum 
$
\mathcal{CQT}\simeq \mathcal W_1 \oplus \mathcal F.
$
Since both $\mathcal W_1$ and $\mathcal F$ are Banach spaces, the composition of the $1$-norm of $\mathbb R^2$ with the vector valued function $T(a)+E\to (\norm{a}\ww,\norm{E}\f)$ makes $\mathcal W_1 \oplus \mathcal F$ a complete metric space.

\end{proof}

Since $\mathcal{CQT}$ is a normed matrix algebra, if $A\in\mathcal{CQT}$ and $B$ is an infinite matrix such that $BA=AB=I$,   then $B\in\mathcal{CQT}$. In the next section we represent the inverse matrix of an infinite Toeplitz matrix $T(a)$ in terms of the Wiener-Hopf factorization of $a(z)$.

\subsection{Matrix inversion}
Assume that $a(z)\in\mathcal W_1$  does not vanishes on the unit circle and its winding number is zero, so that in view of Theorem \ref{th:whf} there exists the canonical Wiener-Hopf factorization $a(z)=u(z)\ell(z)$. From this factorization we deduce the following matrix factorization
\[
T(a)=T(u)T(\ell),
\]
where $T(\ell)$ is lower triangular and $T(u)$ is upper triangular. Since $u(z)$ and $\ell(z^{-1})$ do not vanish in the  unit disk,
 the functions $u(z)$ and $\ell(z)$ have inverse in $\mathcal W_1$, by Theorem \ref{th:1}  are such that $T(u)T(u^{-1})=T(u^{-1})T(u)=I$, and
$T(\ell)T(\ell^{-1})=T(\ell^{-1})T(\ell)=I$, so that
\[
T(a)^{-1}=T(\ell)^{-1}T(u)^{-1}=T(\ell^{-1})T(u^{-1}).
\]

In view of Lemma \ref{lem:2}, we have 
\begin{equation}\label{eq:inv}
T(a)^{-1}=T(a^{-1})-H((\ell^{-1})^-)H((u^{-1})^+)=T(a^{-1})-H(\ell^{-1})H(u^{-1})\in\mathcal{CQT}.
\end{equation}
That is,  a semi-infinite Toeplitz matrix associated with a symbol $a(z)\in\mathcal W_1$,  with null winding number, which does not annihilates in $\mathbb T$, is invertible and its inverse is a CQT-matrix.

This fact, together with the available algorithms to compute the Wiener-Hopf factorization of $a(z)$, enables us to implement the inversion of CQT-matrices in a very efficient manner.
We will see this in the next section.

\section{CQT matrix arithmetic}\label{sec:arith}
The properties that we have described in the previous sections imply that any finite computation which takes as input a set of CQT-matrices and that performs matrix additions, multiplications, inversions, and multiplications by a scalar, generates results that belong to $\mathcal{CQT}$. If the computation can be carried out with no breakdown, say caused by singularity, then the output still belongs to $\mathcal{CQT}$. 

This observation makes it possible to compute functions of semi-infinite CQT-matrices in an efficient way or to solve quadratic matrix equations where the coefficients are CQT-matrices.
In order to do that, we have to provide a simple and effective way of representing, up to an arbitrarily small error, CQT-matrices by means of a finite number of parameters. This is done in this section.

Given a QT-matrix $A=T(a)+E_a$, since the symbol $a(z)$ belongs to the Wiener class, 
and since the correction matrix $E_a$ has entries with finite sum of their
moduli, we may write $A$ through its
{\em truncated form} $\wt A=\trunc(A)$. That is, for any $\epsilon>0$
there exist integers $n_-$, $n_+$, $k_-$, $k_+$ such that
\begin{equation}\label{eq:represent}
\begin{split}
&A=\wt A+\mathcal E_a,\quad \|\mathcal E_a\|\qt\le\epsilon,\\
&\wt A=T(\wt a)+\wt E_a,\\
&\wt a(z)=\sum_{i=-n_-}^{n_+}a_iz^i,
\end{split}\end{equation}
where  $\wt E_a=(\wt e_{i,j})$, is such that $\wt e_{i,j}=e_{i,j}$
for $i=1,\ldots,k_-$, $j=1,\ldots,k_+$, while $\wt e_{i,j}=0$
elsewhere.

In this way, we can approximate any given QT-matrix $A$, to any desired
precision, with a CQT-matrix $\wt A$ where the Toeplitz part is banded
and the correction $\wt E_a$ has a finite dimensional nonzero part. The CQT-matrix $\wt A$
can be easily stored with a finite number of memory locations. The
``finite approximation'' $\wt A$ of a QT-matrix $A$ is the computational
counterpart with which we are going to work in practice.

Observe that, if $A\in\mathcal{CQT}$ and the symbol $a(z)$ is analytic, for the exponential decay of the coefficients $|a_i|$, the values of $n_\pm$ are $O(\log\epsilon^{-1})$. Concerning the values of $k_\pm$, unless we make additional assumptions on the decay of the entries $|e_{i,j}|$ as $i,j$ tend to infinity, the values that $k_\pm$ can assume are as large as $1/\epsilon$.  Think for instance to the case where $e_{i,j}=1/(i+j)^p$ for $p>2$ where $k_\pm$ are of the order of $1/\epsilon^{p-1}$. The same qualitative bounds hold for the coefficients $a_i$ if we simply assume that $a(z)\in\mathcal W_1$.

Here and in the sequel, we do not care much to give {\em a priori} bounds to the values of $n_\pm$ and $k_\pm$ since these values can be determined automatically at run time during the computation.

Another observation concerns the truncated correction $\wt E_a$. In fact, from the
computational point of view, it is convenient to express the matrix
$\wt E_a$ by means of a factorization of the kind $\wt E_a=F_aG_a^T$, where  matrices $F_a$ and $G_a$ have a number of columns given by the
rank of $\wt E_a$ and infinitely many rows. In this way, in presence of low-rank corrections, the storage is reduced together with the computational cost for performing matrix arithmetic. This representation in product form can be obtained by means of SVD up to some error which can be controlled at run time and which can be included in $\mathcal E_a$. Observe also that the truncation operates both on the function $a(z)$ and in the correction $E_a$ by means of compression.

In the following, we represent a QT-matrix $A=T(a)+E_a$ in the form \eqref{eq:represent}
with
$\wt E_a=F_aG_a^T$ where $F_a$ has $f_a$ nonzero rows and $k_a$ columns, $G_a$ has $g_a$ nonzero rows and $k_a$ columns, and the error $\mathcal E_a$ has a sufficiently small norm. 
This way, $\wt E_a$ has $f_a$ nonzero rows, $g_a$ nonzero columns and rank at most $k_a$. 

With this notation we may easily implement the operations of addition, subtraction, multiplication and inversion of two  CQT-matrices $\wt A$, $\wt B$ which are the truncated representations of two QT matrices $A$ and $B$  i.e.,
\[\begin{split}
& A=\wt A+\mathcal E_a,\quad \wt A=\trunc(A)=T(\wt a)+\wt E_a\\
&B=\wt B+ \mathcal E_b,\quad \wt B=\trunc(B)=T(\wt b)+\wt E_b,
\end{split}
\]
denote by $\star$ any arithmetic operation, define $C=A\star B$, $\widehat C=\wt A\star \wt B$ 
and $\widetilde C=\hbox{trunc}(\widehat C)$. 

We define {\em total error} in the operation $\star$ as $\mathcal E^{tot}_c=C-\widetilde C$, the {\em local error} as $\mathcal E^{loc}_c=\widehat C-\widetilde C$ and the {\em inherent error} as $\mathcal E^{in}_c= C-\widehat C$, so that $\mathcal E^{tot}_c=\mathcal E^{in}_c+\mathcal E_c^{loc}$.
Observe that the inherent error is the result of $\mathcal E_a$ and $\mathcal E_b$ through the performed matrix operation, the local error is generated by the truncation of the matrix arithmetic operation $\wt A\star \wt B$, while the total error is the sum of the two errors.
Formally, these errors behave like the inherent error and the round-off error in the standard floating point arithmetic.

In our study we do not analyze the growth of the inherent error in each arithmetic operation, but rather we limit ourselves to operate the truncation and compression  in such a way that the norm of the local error  is bounded by a given value $\epsilon$, say the machine precision. 
Moreover, we do not consider the errors generated by the floating point arithmetic.

\subsection{Addition}
Let $A=\wt A+\mathcal E_a$ and  $B=\wt B+\mathcal E_b$ be CQT matrices where $\wt A=T(\wt a)+\wt E_a$, $\wt B=T(\wt b)+\wt E_b$ with $\wt a(z)$, $\wt b(z)$ Laurent polynomials of degrees $n_a^\pm$ and $n_b^\pm$ respectively, and $\wt E_a=F_aG_a^T$, $\wt E_b=F_bG_b^T$.

If $A$ and $B$ have the above representation, then, for the matrix $C=A+B$ we have the representation
\[C=\wt A+\wt B+\mathcal E_a+\mathcal E_b,
\]
from which we deduce that the inherent error is $\mathcal E^{in}_c=\mathcal E_a+\mathcal E_b$.
On the other hand, concerning $\widehat C=\wt A+\wt B$ we have
\[
\widehat C=T(\wt a+\wt b)+\wt E_a+\wt E_b,
\] 
where $\wt a(z)+\wt b(z)$ is a Laurent polynomial of degrees $n_c^-=\max(n_a^-,n_b^-)$,
$n_c^+=\max(n_a^+,n_b^+)$.
 while 
\[\begin{split}
&E_c=\wt E_a+\wt E_b=F_cG_c^T,\\
&F_c=[ F_a , F_b],\quad G_c=[ G_a , G_b],
\end{split}
\]
 where $f_c=\max(f_a,f_b)$ and $g_c=\max(g_a,g_b)$ are the number of nonzero rows of $F_c$ and $G_c$, respectively,  and $k_c=k_a+k_b$ is the number of columns of $F_c$ and $G_c$.

The Laurent polynomial $\wt a(z)+\wt b(z)$ can be truncated and replaced by a Laurent polynomial $\wt c(z)$ of possibly less degree. Also  
the value of $k_c$, can be reduced and the matrices $F_c$, $G_c$ can be compressed, by using a compression technique which guarantees a local error with norm bounded by a given $\epsilon$. This technique, based on computing SVD and QR factorization is explained in the next section. Denoting by  $\wt F_c$, $\wt G_c$ the matrices obtained after compressing $F_c$ and $G_c$, respectively, we have

\[
\wt C=\hbox{trunc}(\widehat C)=T(\wt c)+\wt E_c+\mathcal E^{loc}_c,\quad \wt E_c=\wt F_c\wt G_c^T,
\]
where $\mathcal E^{loc}_c$ denotes the local error due to truncation and compression, i.e.
$\mathcal E^{loc}_c=\wt A+\wt B-\trunc(\wt A+\wt B)$.
This way we have
\[
A+B=T(\wt c)+\wt E_c+\mathcal E^{loc}_c+\mathcal E^{in}_c.
\]

\subsection{Multiplication}
A similar expression holds for multiplication. For the product $C=AB$ we have the equation
\[
AB=\wt A\wt B+\wt A\mathcal E_b+\mathcal E_a\wt B+\mathcal E_a\mathcal E_b
\]
from which we deduce that the inherent error is $\mathcal E^{in}_c=\wt A\mathcal E_b+\mathcal E_a\wt B+\mathcal E_a\mathcal E_b$. Moreover we have
\[\begin{split}
\widehat C=\wt A\wt B=&T(\wt a)T(\wt b)+T(\wt a)\wt E_b+\wt E_a T(\wt b)+\wt E_a\wt E_b\\
=& T(\wt a\wt b)-H(\wt a^-)H(\wt b^+)+T(\wt a)\wt E_b+\wt E_a T(\wt b)+\wt E_a\wt E_b\\
=:&T(\wt a\wt b)+E_c.
\end{split}
\]

Observe that, since $\wt a^-(z)$ and $\wt b^+(z)$ are polynomials, the matrices $H(\wt a_-)$ and $H(\wt b_+)$ have a finite number of nonzero entries. Therefore, we may factorize the product $H(\wt a^-)H(\wt b^+)$ in the form $FG^T$.
Thus, we find that the matrix $E_c$ can be written as $E_c=F_c G_c^T$ where
\[
F_c=[F,T(\wt a)F_b,F_a],\quad G_c=[G,G_b,T(\wt b)^TG_a+G_b(F_b^TG_a)].
\]
This provides the finite representation of the product $\widehat C=\wt A \wt B$ where
$n^-_c= n^-_a+n^-_b$, $n^+_c= n^+_a+n^+_b$, 
$f_c=\max(f_b+n^-_a,f_a)$, 
$g_c=\max(n^+_b,g_b,g_a+n^-_b)$,
and $k_c=k_a+k_b+n^+_b$. 

Also in this case  we may apply a compression technique, based on SVD for reducing the memory storage of the correction and for reducing the degree of the Laurent polynomial $\wt a(z)\wt b(z)$. Operating in this way, we introduce a local error 
$\mathcal E^{loc}_c=\wt A\wt B-\trunc(\wt A\wt B)$. Denoting by $\wt c(z)$ the truncation of the Laurent polynomial $\wt a(z)\wt b(z)$ and with $\wt F_c\wt G^T_c$  the compression of $F_cG_c^T$, we have
\[
\widehat C=\wt A\wt B = T(\wt c) + \wt F_c\wt G_c^T+\mathcal E^{loc}_c.
\]
This way we have
\[
C=AB=T(\wt c) + \wt F_c\wt G_c^T+\mathcal E^{loc}_c+\mathcal E^{in}_c,
\]
which expresses the result $C$ of the multiplication in terms of the approximated value $\wt C=T(\wt c) +\wt E_c$, the local error $\mathcal E^{loc}_c$ and the inherent error
$\mathcal E^{in}_c$.
The overall error is given by $\mathcal E_c=\mathcal E^{loc}_c+ \mathcal E^{in}_c$.

\subsection{Matrix inversion}
It is worth paying a particular attention to the operation of matrix inversion since it is less immediate than multiplication and addition.

First, we consider the problem of inverting the matrix $A=T(a)$, i.e., we assume that $E_a=0$. The general case will be treated afterwords.

Recall that, if $a(z)\in\mathcal W_1$ does not vanish in the unit circle  and if it has a zero winding number, then  Theorem \ref{th:goh} implies that
the matrix $T(a)$ is invertible and, in view of Theorem \ref{th:whf}, there exists the canonical Wiener-Hopf factorization $a(z)=u(z)\ell(z)$ so that \eqref{eq:inv} holds.
Thus, a finite representation of $A^{-1}$ is obtained by truncating the Laurent series of $1/a(z)$ to a Laurent polynomial and by approximating the Hankel matrices $H((\ell^{-1})^-)$ and $H((u^{-1})^+)$ by means of matrices having a finite number of nonzero entries, an infinite number of rows and the same finite number of columns.  The latter operation can be achieved by truncating the power series $\ell^{-1}(z)$ and $u^{-1}(z)$ to polynomials and by numerically compressing the product of the Hankel matrices obtained this way. This operation can be effectively performed by reducing the Hankel matrices to tridiagonal form by means of Lanczos method with orthogonalization. This procedure takes advantage of the Hankel structure since the matrix-vector product can be computed by means of FFT in $O(n\log n)$ operations where $n$ is the size of the Hankel matrix. The advantage of this compression is that the cost grows as $O(r^2n\log n)$ where $r$ is the numerical rank of the matrix.

If $a(z)$ is analytic in the annulus 
$\mathbb A(r_a,R_a)=\{z\in\mathbb C:\quad r_a<|z|<R_a\}\supset\mathbb T$, 
then its coefficients have an exponential decay so that  
$|a_i^+|\le\gamma \lambda_+^i$, $|a_i^-|\le\gamma \lambda_-^i$,
$|u_i|\le\gamma \lambda_+^i$, $|\ell_i^-|\le\gamma \lambda_-^i$,
for some positive $\gamma$ and for $1/R_a<\lambda_+<1$, $r_a<\lambda_-<1$. Thus, we find that for the truncated approximation of the matrix $A$ the values of $n^+$, $n^-$, $f$, $g$ are bounded by $\log(\gamma^{-1}\epsilon^{-1})/\log (\lambda_\pm^{-1})$.

Performing numerical experiments it turns out that the singular values of the principal submatrices of the Hankel matrices $H(\ell^-)$ and $H(u^+)$ associated with power series having coefficients with an exponential decay, have an exponential decay themselves. So that also the truncation on the value of the numerical rank $k$ of $H(\ell^-)H(u^+)$ can be performed efficiently.

The analysis of the inherent error due to inversion is related to the analysis of the condition number of semi-infinite Toeplitz matrices. We do not carry out this analysis, we refer the reader to the books \cite{bottcher2005spectral}, \cite{bottcher2012introduction} on this regard.

Now consider the more general case of the matrix $A=T(a)+F_aG_a^T$ which we assume already in its truncated form. Assume $T(a)$ invertible and write $A=T(a)(I+T(a)^{-1}F_aG_a^T)$.
Denoting for simplicity $U=T(u)$, $L=T(\ell)$ we have
\[\begin{split}
&(T(a)+F_aG_a^T)^{-1}=  
T(a)^{-1}-
L^{-1}(U^{-1}F_a)Y^{-1}(G_a^TL^{-1})U^{-1},\\
 &Y=I+G_a^TL^{-1}U^{-1}F_a,
\end{split}\]
where $Y$ is a finite matrix which is invertible if and only if $A$ is invertible.
This way, the algorithm for computing $A^{-1}$ in its finite $QT$-matrix representation is given by the following steps:
\begin{enumerate}
\item compute the spectral factorization $a(z)=u(z)\ell(z)$;
\item compute the coefficients of the power series $\wt u(z)=1/u(z)$ and $\wt \ell(z)=1/\ell(z)$, so that $L^{-1}=T(\wt \ell)$, $U^{-1}=T(\wt u)$;
\item represent the matrix $H=L^{-1}U^{-1}$ as $T(c)+F_hG_h^T$, where $c(z)=\wt \ell(z)\wt u(z)$ by means of  Theorem \ref{th:1};
\item compute the products: $G_1=T(\wt\ell)G_a$, $F_1=T(\wt u)F_a$;
\item compute $Y=I+G_1^TF_1$, $F_2=F_1Y^{-1}$, $F_3=T(\wt\ell)F_2$, $G_2=T(\wt u)G_1$;
\item output the coefficients of $c(z)$ and the matrices $F_c=[F_h, F_3]$, $G_c=[G_h,G_2]$.
\end{enumerate}

For computing the spectral factorization of $a(z)$ we rely on the algorithm of \cite{bfgm} which employs  evaluation/interpolation techniques at the Fourier points.

\subsection{Compression}
Given the matrix $E$ in the form $E=FG^T$ where $F$ and $G$ are matrices of size $m\times k$ and $n\times k$, respectively, we aim to reduce the size $k$ and to approximate $E$ in the form $\wt F\wt G^T$ where $\wt F$ and $\wt G$ are matrices of size $m\times \wt k$ and $n\times \wt k$, respectively, with $\wt k<k$.

We use the following procedure. Compute the pivoted (rank-revealing) QR factorizations
$F= Q_fR_fP_f$, $G= Q_gR_gP_g$, where $P_f$ and $P_g$ are permutation matrices, $Q_f$ and $Q_g$ are orthogonal and $R_f$, $R_g$ are upper triangular; remove the last negligible rows from the matrices $R_f$ and $R_g$,   remove the corresponding columns of $Q_f$ and $Q_g$. 
In this way we obtain matrices $\hat R_f$, $\hat R_g$, $\hat Q_f$, $\hat Q_g$ such that,  up to within a small error, satisfy the equations $F= \hat Q_f\hat R_fP_f$, $G= \hat Q_g\hat R_gP_g$. Then, in the factorization $FG^T=\hat Q_f(\hat R_fP_fP_g^T\hat R_g^T)\hat Q_g^T$, compute the SVD of the matrix in the middle
$ \hat R_fP_fP_g^T\hat R_g^T=U\Sigma V^T$, and replace $U,\Sigma$ and $V$ with matrices $\hat U$, $\hat\Sigma$, $\hat V$, obtained by removing the singular values $\sigma_i$ and the corresponding singular vectors if   $\sigma_i<\epsilon\sigma_1$, where $\epsilon$ is a given tolerance. In output, the matrices $\wt F=\hat Q_f \hat U\hat \Sigma^{1/2}$, $\wt G=\hat Q_g\hat V\hat \Sigma^{1/2}$ are delivered.

\subsection{The Matlab code}
The arithmetic operations on CQT-matrices have been implemented in Matlab.
The package can be obtained upon request from the authors.
 It includes the functions
{\tt qt\_add}, {\tt qt\_mul}, {\tt qt\_inv}, {\tt qt\_compress} for performing matrix arithmetic and compression. A CQT-matrix $A$ is stored by means of the variables {\tt am, ap, aF, aG}, where {\tt am} and {\tt ap} are the vectors containing the coefficients of the Laurent polynomial $a(z)=\sum_{i=-h}^ka_iz^i$ so that
{\tt am = }$(a_0,a_{-1},\ldots,a_{-h})$,
{\tt ap = }$(a_{0},a_1,\ldots,a_k)$, the variables {\tt aF} and {\tt aG} contain the values of the non-negligible entries in the correction matrices $F$ and $G$, respectively.

In each function, after performing an arithmetic operation, the compression of the matrices $F$ and $G$ is applied.

\section{Solving certain semi-infinite quadratic matrix equations by means of Cyclic Reduction}\label{sec:qbd}
In the analysis of certain QBD queueing processes like the tandem Jackson queue \cite{jackson1957networks} or bi-dimensional random walks in the quarter plane \cite{miyazawa2011light,kobayashi2012revisit}, one has to find the invariant probability vector  of a stochastic process with a discrete two-dimensional state space. The two coordinates of the latter ---usually called level and phase---  are both countably infinite. Typically, the allowed transitions from a state are limited to a subset of the adjacent states. Moreover, the probability of a certain transition is homogeneous in time and ---except for some boundary conditions--- depends only on the distance between the departure and the arrival. This makes the model representable with a generator  of the kind
\[
Q=\begin{bmatrix}\widehat A_0&\widehat A_1\\
A_{-1}&A_0&A_1\\
      &A_{-1}&A_0&A_1\\
      &&\ddots&\ddots&\ddots\end{bmatrix},\quad A_i,\widehat A_i\in\mathcal{CQT}.
\]

The matrix analytic methods, designed in this framework for finding the invariant distribution, require to find
the minimal non negative solutions $G$ and $R$ of the semi-infinite matrix equations
\begin{equation}\label{eq:mateq}
A_1X^2+A_0X+ A_{-1}=0,\quad X^2A_{-1}+XA_0+A_1=0,
\end{equation}
respectively.
It can be proved that, under very mild assumptions, the minimal non-negative solutions of the above equations exist and are unique.
We refer to the books \cite{bini2005numerical},
\cite{lr:book},
\cite{neuts1981matrix}, for more details.

When the blocks $A_i$s are finite, one of the most reliable and fast algorithms for performing this computation  is the Cyclic Reduction (CR) \cite{bini2009cyclic,hockney1965fast,buzbee1970direct}. This is an iterative method based on generating the following matrix sequences
\begin{equation}\label{cr}
\begin{split}
&A_0^{(h+1)}=A_0^{(h)}-A_1^{(h)}S^{(h)}A_{-1}^{(h)}-A_{-1}^{(h)}S^{(h)}
A_1^{(h)},\quad S^{(h)}=(A_0^{(h)})^{-1},\\
&A_1^{(h+1)}=-A_1^{(h)}S^{(h)}A_1^{(h)},\quad A_{-1}^{(h+1)}=-A_{-1}^{(h)}S^{(h)}A_{-1}^{(h)},\\
& \wt A^{(h+1)}=\wt A^{(h)}-A_{-1}^{(h)}S^{(h)}A_1^{(h)},
\quad
\widehat A^{(h+1)}=\widehat A^{(h)}-A_{1}^{(h)}S^{(h)}A_{-1}^{(h)},
\end{split}
\end{equation}
for $h=0,1,2\ldots$,
with $A_0^{(0)}=\wt A^{(0)}=\widehat A^{(0)}=A_0$, $A_1^{(0)}=A_1$, $A_{-1}^{(0)}=A_{-1}$. 
The sequences
\begin{equation}\label{eq:Gh}
G^{(h)}:=-(\wt A^{(h)})^{-1}A_{-1},\quad R^{(h)}:=-A_1(\widehat A^{(h)})^{-1}
\end{equation}
 converge to the minimal non-negative solutions $G$ and $R$ of the matrix equations \eqref{eq:mateq}. 

 These convergence properties are valid also in the case where the blocks $A_{-1}$, $A_0$, $A_1$ are semi-infinite where convergence holds component-wise. We refer the reader to \cite{lato:varese} for more details.

The arithmetic developed in Section \ref{sec:arith} paves the way to the use of CR when $A_i\in\mathcal{CQT}$, $i=-1,0,1$. Observe that, since $\mathcal {CQT}$ is an algebra, all the matrices generated by CR belong to $\mathcal {CQT}$. Moreover, the Toeplitz part of these matrices have associated symbols $a_{-1}^{(h)}(z)$, $a_0^{(h)}(z)$, 
$a_1^{(h)}(z)$,  $\wt a^{(h)}(z)$, $\widehat a^{(h)}(z)$, which satisfy the same recurrence equations as \eqref{cr}. More precisely we have
 the {\em scalar} functional relations
\[
\begin{split}
&a^{(h+1)}_0(z)= a^{(h)}_0(z)-2a^{(h)}_1(z)a^{(h)}_{-1}(z)/a^{(h)}_0(z),\\
&a^{(h+1)}_1(z)=-a^{(h)}_1(z)^2 /a^{(h)}_0(z),\quad a^{(h+1)}_{-1}(z)=-a^{(h)}_{-1}(z)^2 /a^{(h)}_0(z),\\
&\wt a^{(h+1)}(z)=\wt a^{(h)}(z)-a^{(h)}_{1}(z) a^{(h)}_{-1}(z)/a^{(h)}_0(z),
\end{split}
\]
with $h=0,1,\ldots$, where $a_i^{(0)}(z)=a_i(z)$, $i=-1,0,1$  and $\wt a^{(0)}(z)=a_0(z)$. Observe that since all the quantities in the above recurrence are scalar functions, they commute so that $\widehat a^{(h)}(z)$ coincides with 
$\wt a^{(h)}(z)$.

As pointed out in \cite{bgm:laa}, \cite{bini2009cyclic}, in the scalar case CR reduces to the celebrated Graeffe iteration whose properties have been investigated in \cite{ostro1}. Thus,
in order to analyze the convergence of the sequences defined above, we
rely on the convergence properties of the Graeffe iteration
applied to quadratic polynomials. In particular, we
know that if, for a given $z\in\mathbb T$ the polynomial
$p_z(x):=a_1(z)x^2+a_0(z)x+a_{-1}(z)$ associated with the triple $(a_{-1}(z),a_0(z),a_1(z))$, has one root inside the unit disk and one
root outside, then the sequence $-(a_{-1}(z)/ \wt a^{(h)}(z))$ has a limit $g(z)$ which coincides with the root of the polynomial $p_z(x)$ inside the unit disk.

The following theorem provides mild conditions which ensure the above properties, and are generally satisfied in the applications.

\begin{theorem}\label{thm:delta}Let $a_i(z)=a_{i,-1}z^{-1}+a_{i,0}+a_{i,1}z$, for $i=-1,0,1$, be such that
$\sum_{i,j=-1}^1 a_{i,j}=0$, $a_{0,0}< 0$, $a_{i,j}\ge 0$, otherwise. If 
\begin{itemize}
\item[\rm (i)]  $ a_{-1,0}>0$ or $a_{1,0}>0$,
\item[\rm (ii)] $a_{ij}\neq 0$ for at least a pair $(i,j)$, with $j\ne 0$,
\end{itemize} then for any $z\in\mathbb T$, $z\ne 1$,  the quadratic polynomial  $p_z(x)=a_1(z)x^2+a_0(z)x+a_{-1}(z)$, has a root of modulus less than 1 and a root of modulus greater than 1.
\end{theorem}
\begin{proof} Without loss of generality we may assume that the entries $a_{i,j}$ belong to the interval $[-1,1]$. If not, we may scale equation \eqref{eq:mateq} by a suitable constant and reduce it to this case.
As a first step we show that there are no roots of modulus $1$. Assume by contradiction that $x$ is a root of modulus 1. Obviously,  we have $p_z(x)=0$ if and only $p_z(x)+x=x$. Observe that, if $z\in\mathbb T$,  the left hand-side of the previous equation is a convex combination of the points in the discrete set $\mathcal C_{x,z}:=\{x^iz^j,\ i=0,1,2,\ j=-1,0,1\}\subset\mathbb T$. If $z\ne 1$, condition (i) and the fact that $-1\le a_{0,0}< 0$ ensure that the convex combination involves at least two different points of the unit circle, either $x$ and $1$ or $x$ and $x^2$.  Therefore, this convex combination $p_z(x)+x$ is equal to a point which belongs to the interior of the unit disc. This contradicts the fact that $|p_z(x)+x|=|x|=1$. This argument excludes roots on $\mathbb T$ for  $z\in\mathbb T\setminus\{1\}$. We conclude by showing that there is exactly one root of modulus less than $1$.  In order to prove this, we first show that $|a_0(z)|>|a_{-1}(z)+a_1(z)|$ holds for any $z\in\mathbb T\setminus \{1\}$. Therefore, by applying the Rouch\'e Theorem one finds that the functions $f(x)=a_0(z)x$ and $p_z(x)$ have the same number of zeros in the open unit disc. To prove the inequality $|a_0(z)|>|a_{-1}(z)+a_1(z)|$ we observe that
\begin{align*}
| a_{ 0,-1}z^{-1}+ a_{ 0,0}+ a_{0,1}z|&\ge|a_{ 0,0}|-| a_{ 0,-1}z^{-1}|-| a_{0,1}z|=-a_{0,0}-a_{ 0,-1}-a_{0,1}\\&= a_{-1,-1}+ a_{ -1,0}+ a_{ -1,1}+a_{1,-1}+ a_{ 1,0}+ a_{ 1,1}
\\
&\ge| a_{-1,-1}z^{-1}+ a_{ -1,0}+ a_{ -1,1}z+a_{1,-1}z^{-1}+ a_{ 1,0}+ a_{ 1,1}z|
\end{align*}
where at least one of the two above inequalities is strict because of condition (ii).
\end{proof}

\begin{corollary}
Under the conditions of Theorem~\ref{thm:delta}, if $a_1(z)\ne 0$ for any $z\in\mathbb T$ and $a_{-1}(1)\neq a_1(1)$, then $g(z)=\lim_h -a_1(z)/\wt a^{(h)}(z)$ is an analytic function.  
\end{corollary}
\begin{proof}
We recall that the roots of a monic polynomial are analytic functions of the coefficients, on the set where the polynomial has not multiple roots \cite{Brillinger}.  Thus, in
order to prove the analyticity of $g(z)$, it is sufficient to show that that $p_z(x)$ has no multiple root  $\forall z\in\mathbb T$.
 This follows from  Theorem~\ref{thm:delta} if 
 $z\in\mathbb T\setminus\{1\}$. Moreover, observe that for $z=1$, $p_1(x)$ has roots $1$ and $\frac{a_{-1}(1)}{a_1(1)}$ where the latter is real, non negative and different from 1 by assumption.
\end{proof}

With the information that we have collected so far, we cannot yet say if the matrix $G$ belongs to $\mathcal CQT$. In fact, in principle, writing $G=T(g)+E_g$, it is not ensured that $\|E_g\|\f<\infty$. The boundedness of $\|E_g\|\f$ can be proved if $E_g$ has all entries with the same sign. This analysis is part of the subject of our future research. On this regard, it is worth citing the paper \cite{guy} where, relying on probabilistic arguments, it is proved that the matrices $G$ and $R$ asymptotically share the Toeplitz structure.

\subsection{Numerical validation}
In order to validate our analysis, we consider ten instances of the two-node Jackson network, analyzed in \cite{Motyer-Taylor}. In details, we assume
\begin{align*}
A_{-1}&=\begin{bmatrix}
(1-q)\mu_2&q \mu_2\\
&(1-q)\mu_2&q \mu_2\\
&&\ddots&\ddots
\end{bmatrix},\\
A_0&=\begin{bmatrix}
-(\lambda_1+\lambda_2+\mu_2)&\lambda_1\\
(1-p)\mu_1&-(\lambda_1+\lambda_2+\mu_1+\mu_2)&\lambda_1\\
&\ddots&\ddots&\ddots
\end{bmatrix},\\
A_1&=\begin{bmatrix}
\lambda_2\\
p\mu_1&\lambda_2\\
&\ddots&\ddots
\end{bmatrix},
\end{align*}
where the parameters $p,q,\lambda_1,\lambda_2,\mu_1,\mu_2$ are chosen according to Table~\ref{tab:param1}. These examples are also studied in \cite{Sakuma-Miyazawa} where it is shown the bad effect of truncation in approximating the stationary distribution. Different decay properties of the invariant probability distribution correspond to the different values of the parameters.

We have applied CR in all the 10 cases and computed the minimal non-negative solution $G$ represented in the CQT form as $T(g)+U_gV_g^T$. In the results of the tests that we have performed, we report, besides the CPU time in seconds, also the
norm of the residual error $E=A_1 G^2+A_0 G+A_{-1}$ where we used both the infinity norm $\|E\|_\infty$ and the CQT norm $\|E\|\aqt$.

In order to analyze the intrinsic complexity of the problem, we also report
the
 band width of the matrix $T(g)$, that is the number of non-negligible coefficients of the Laurent series $\sum_{i\in\mathbb Z}g_iz^i$, the number of the nonzero rows of the matrices $U_g$ and $V_g$ and the number of their columns that is their rank.

 All this information is reported in Table \ref{tab:result1}.  We may
 observe that a high CPU time, like for instance in the case of
 Problem 7, corresponds to large values of the band width in the
 matrix $T(g)$ or to large sizes of the correction. The large values
 of these two components of the CQT representation of $G$ imply that
 the entries $g_{i,j}$ have a low decay speed as $i,j\to \infty$.

\begin{table}
         \centering
         \begin{filecontents*}{parameters.txt}
1	1	0	1.5	2	1	0
2	1	0	2	1.5	1	0
3	0	1	1.5	2	0	1
4	0	1	2	1.5	0	1
5	1	1	2	2	0.1	0.8
6	1	1	2	2	0.8	0.1
7	1	1	2	2	0.4	0.4
8	1	1	10	10	0.5	0.5
9	1	5	10	15	0.4	0.9
10	5	1	15	10	0.9	0.4
\end{filecontents*}
\pgfplotstabletypeset[
   	      columns={0,1,2,3,4,5,6}, 
   	      columns/0/.style={
   		      column name = Case	
   	      },
columns/1/.style={
column name = $\lambda_1$	
},
   	      columns/2/.style={
   	         		      column name = $\lambda_2$	
   	         	      },
   	      columns/3/.style={
   	         		      column name = $\mu_1$	
   	         	      },
   	      columns/4/.style={
   	         		      column name = $\mu_2$	
   	         	      },   
   	      columns/5/.style={
   	         	          column name = $p$	
   	         	      },
   	      columns/6/.style={
   	         	          column name = $q$	
   	         	      }   	      	             ]{parameters.txt}
         \caption{Parameters values of the test examples for the two node Jackson tandem network}
         \label{tab:param1}
       \end{table}

\begin{table}
         \centering
         \begin{filecontents*}{jackson.txt}
1	2.61e+00	8.63e-16	5.98e-13	561	541	138	8
2	2.91e+00	1.49e-15	7.88e-13	561	555	145	8
3	2.88e-01	1.11e-16	2.67e-14	143	89	66	8
4	2.32e+00	6.77e-16	6.00e-13	463	481	99	9
5	4.77e-01	1.23e-15	1.07e-13	233	108	148	9
6	7.96e+00	1.92e-14	6.65e-13	455	462	153	10
7	2.90e+01	4.29e-15	6.87e-12	1423	1543	247	13
8	1.01e+00	1.14e-15	4.34e-13	366	348	40	6
9	3.01e-01	5.44e-16	2.48e-14	157	81	86	8
10	1.25e+00	1.09e-15	3.40e-14	268	241	107	8
\end{filecontents*}
\pgfplotstabletypeset[          
   	      columns={0,1,2,3,4,5,6,7}, 
   	      columns/0/.style={
   		      column name = Case	
   	      },
columns/1/.style={
column name = CPU time,
postproc cell content/.style={
   	      		/pgfplots/table/@cell content/.add={}{}
   	      		}
   	      	},
   	      columns/2/.style={
   	         		      column name = $Res_{\infty}$	
   	         	      },
   	      columns/3/.style={
   	         		      column name = $Res\aqt$ },   	         	      
   	      columns/4/.style={
   	         		      column name = Band 
   	         	      },
   	      columns/5/.style={column name =  Rows
}, 
columns/6/.style={   	         	         	column name = Columns	
},   
   	      columns/7/.style={
   	         	          column name =  Rank	
   	         	      }  	      	             ]{jackson.txt}
         \caption{Features of the computed solutions by means of CR}
         \label{tab:result1}
       \end{table}

\section{Conclusion}\label{sec:conc}
We have introduced the class of semi-infinite quasi-Toeplitz matrices and proved that it is a Banach space with a suitable norm. Then we have considered the subspace formed by quasi-Toeplitz matrices associated with a continuous symbol $a(z)$ such that $a'(z)\in\mathcal W$, and proved that it is a Banach algebra where the norm is sub-multiplicative. These properties have been used to define a matrix arithmetic on the algebra of semi-infinite CQT matrices. These tools have been used to design methods for solving quadratic matrix equations with semi-infinite matrix coefficients encountered in the analysis of a class of QBD stochastic processes.

\section*{Acknowledgments}
Dario Bini wishes to thank Bernhard Beckermann and Daniel Kressner for very helpful discussions on topics related to the subject of this paper.

 \bibliographystyle{abbrv}
 
\end{document}